\documentclass[11pt]{amsart}
\usepackage{geometry}
\usepackage{graphicx}
\usepackage{amssymb}
\usepackage{multirow}
\usepackage{color}

\newcommand{\hl}[1]{\colorbox{yellow}{#1}}

\newtheorem{theorem}{Theorem}

\newtheorem*{remark*}{Remark}
\newtheorem{corollary}{Corollary}
\newtheorem{definition}{Definition}
\newtheorem{lemma}{Lemma}
\newtheorem{example}{Example}

\usepackage[numbers,sort&compress]{natbib}
\bibpunct{[}{]}{;}{n}{,}{,}

\usepackage[colorlinks=false,bookmarksdepth=0]{hyperref}

\DeclareMathOperator*{\ext}{ext}

\title{Properties of Hamiltonian Variational Integrators}
\author{Jeremy M. Schmitt}
\author{Melvin Leok}
\address{Department of Mathematics, University of California, San Diego, 9500 Gilman Drive \#0112, La Jolla, CA  92093-0112,
USA.}
\email{\texttt{j2schmit@ucsd.edu}, \texttt{mleok@ucsd.edu}}

\begin{document}

\begin{abstract}
Discrete Hamiltonian variational integrators are derived from Type II and Type III generating functions for symplectic maps, and in this paper we establish a variational error analysis result that relates the order of accuracy of the associated numerical methods with the extent to which these generating functions approximate the exact discrete Hamiltonians. We also introduce the notion of an adjoint discrete Hamiltonian, and relate it to the adjoint of the associated symplectic integrator. We show that when constructing discrete Lagrangians and discrete Hamiltonians using the Taylor variational integrator approach, the same underlying one-step method and quadrature rule does not necessarily lead to the same symplectic integrator, and the same observation holds when developing variational integrators based on averaging techniques. Numerical experiments also indicate that the resonance behavior of variational integrators also depend on the type of generating functions used, and we relate this resonance behavior to the ill-posedness of the boundary-value problems used to define the exact discrete Lagrangian and exact discrete Hamiltonian.
\end{abstract}
	
\maketitle

\section{Introduction}
Geometric numerical integration is a field of numerical analysis that develops numerical methods with the goal of preserving geometric properties of dynamical systems (see \cite{HaLuWa2006}). Variational integrators are geometric numerical integrators derived from discretizing Hamilton's principle from classical mechanics (see \cite{MaWe2001}). They have many desirable properties such as symplecticty, momentum-preservation, and near-energy preservation, which results in excellent long-term stability. While the Lagrangian formulation of variational integrators has been thoroughly investigated (see \cite{MaWe2001,LeSh2011_sbvi,LeMaOr2008,LeDiMe2007,LeMaOrWe2003,MaPaSh1998}), only recently has the Hamiltonian formulation of variational integrators been established (see \cite{LaWe2006,LeZh2011}). 
\\

In this paper we will continue the investigation of Hamiltonian variational integrators, and establish theorems on error analysis, symmetry of the method, and provide numerical experiments to elucidate the relative numerical advantages and disadvantages of the Lagrangian and Hamiltonian formulations. In particular, evidence is presented to show that for oscillatory problems the discrete Lagrangian and discrete Hamiltonian variational integrators have differing resonance and conditioning properties. In addition, it is shown that some approximation methods will only yield a symmetric method when derived from a specific type of generating function. The upshot is that the numerical properties of a variational integrator are determined both by the approximation scheme used to construct it and by the type of the generating function being approximated. 

\subsection{Discrete Mechanics}
Lagrangian variational integrators are based on a discrete analogue of Hamilton's principle, and Hamiltonian variational integrators are based on a discrete analogue of Hamilton's phase space variational principle. The fundamental objects in the discretization are generating functions of symplectic maps, and in the Hamiltonian case, they are obtained by approximating the exact Type II generating function associated with a Hamiltonian flow, which we refer to as the exact discrete right Hamiltonian,
\begin{align}
H_d^{+,E}(q_0,p_1)& =  \ext_{\substack{(q, p) \in
C^2([0, T],T^*Q)\\q(0)=q_0, p(T)=p_1}} \left(p_1 q_1 - \int_0^T \left[ p \dot{q}-H(q, p) \right]
dt\right).
\end{align}
This can be viewed as the solution at time $T$ of the Type II Hamilton--Jacobi equation,
\begin{align} \label{type2hj}
\frac{\partial S_2( q_0, p, t)}{\partial t} =H\left(\frac{\partial S_2}{\partial p},p\right),
\end{align}
which more generally describes the Type II generating function which generates the time-$t$ Hamiltonian flow map,
\begin{align}\label{type2sol}
S_2(q_0, p, t)=\ext_{\substack{(q, p) \in
C^2([0, t],T^*Q)\\q(0)=q_0, p(t)=p}}\left (p(t)  q(t) - \int_0^t \left[ p(s) \dot{q}(s)-H(q(s), p(s)) \right] ds\right).
\end{align}
Similarly, the exact discrete left Hamiltonian is given by,
\begin{align}
H_d^{-,E}(p_0,q_1)& =  \ext_{\substack{(q, p) \in
C^2([0, T],T^*Q)\\q(0)=q_0, p(T)=p_1}} \left(-p_0 q_0 - \int_0^T \left[ p \dot{q}-H(q, p) \right]
dt\right).
\end{align}
and it can be viewed as a solution at time $T$ of the Type III Hamilton--Jacobi equation,
\begin{align} \label{type3hj}
\frac{\partial S_3( p_0, q, t)}{\partial t} =H\left(q,-\frac{\partial S_3}{\partial q}\right).
\end{align}
Given discrete Hamiltonians, $H_d^\pm(q_k,p_{k+1})$, the discrete Hamilton's equations are given by,
\begin{align}
q_{k+1}&=D_2H_d^+(q_k, p_{k+1}),\\
p_k&=D_1 H_d^+(q_k, p_{k+1}),\\
\intertext{and,}
q_k&=-D_1 H_d^-(p_k, q_{k+1}),\\
p_{k+1}&=-D_2 H_d^-(p_k, q_{k+1}).
\end{align}
These can also be expressed in terms of the discrete Legendre transformations, $\mathbb{F}^\pm H_d^+:(q_k,p_{k+1})\rightarrow T^*Q$,
\begin{align}
\mathbb{F}^+H_d^+(q_k,p_{k+1})&=(D_2 H_d^+ (q_k, p_{k+1}),p_{k+1}),\\
\mathbb{F}^-H_d^+(q_k,p_{k+1})&=(q_k, D_1H_d^+ (q_k, p_{k+1})),
\end{align}
and $\mathbb{F}^\pm H_d^-:(p_k,q_{k+1})\rightarrow T^*Q$,
\begin{align}
\mathbb{F}^+H_d^-(p_k,q_{k+1})&=(q_{k+1},-D_2 H_d^- (p_k, q_{k+1})),\\
\mathbb{F}^-H_d^-(p_k,q_{k+1})&=(-D_1H_d^- (p_k, q_{k+1}),p_k).
\end{align}
We observe that the Hamiltonian maps $\tilde{F}_{H_d^\pm}:(q_k,p_k) \mapsto (q_{k+1},p_{k+1})$ can be expressed as
\begin{equation}
\tilde{F}_{H^\pm_d}=\mathbb{F}^+H^\pm_d \circ (\mathbb{F}^-H^\pm_d)^{-1}.
\end{equation}

\section{Error Analysis and Symmetric Methods}
\subsection{Error Analysis}
Variational integrators are able to benefit from and adopt many traditional techniques and methods of numerical analysis (see \cite{LeSh2011_sbvi}). This can be largely attributed to the following theorem from \cite{MaWe2001}.
\begin{theorem}[Theorem 2.3.1, \citet{MaWe2001}]
If a discrete Lagrangian, $L_d:Q \times Q \rightarrow \mathbb{R}$, approximates the exact discrete Lagrangian, $L_d^E:Q \times Q \rightarrow \mathbb{R}$ to order $r$, i.e.,
$$L_d(q_0,q_1;h)=L_d^E(q_0,q_1;h) + \mathcal{O}(h^{r+1}),$$
then the discrete Hamiltonian map, $\tilde{F}_{L_d}:(q_k,p_k) \mapsto (q_{k+1},p_{k+1})$, viewed as a one-step method, is order $r$ accurate.
\end{theorem}

Thus, in order to generate a variational integrator of a particular order, one can leverage techniques from numerical analysis with the goal of approximating the exact discrete Lagrangian, then the associated discrete Hamiltonian map yields the variational integrator. We first present the corresponding theorem for discrete Hamiltonian variational integrators, which draws much of its inspiration from the theorem and proof of the above result as detailed in \cite{MaWe2001}.

\begin{theorem}
If a discrete right Hamiltonian, $H^+_d:T^*Q \rightarrow \mathbb{R}$, approximates the exact discrete Hamiltonian, $H_d^{+,E}:T^*Q \rightarrow \mathbb{R}$ to order $r$, i.e.,
$$H^+_d(q_0,p_1;h)=H_d^{+,E}(q_0,p_1;h) + \mathcal{O}(h^{r+1}),$$
and the Hamiltonian is continuously differentiable, then the discrete map, $\tilde{F}^h_{H_d^+}:(q_k,p_k) \mapsto (q_{k+1},p_{k+1})$, viewed as a one-step method, is order $r$ accurate.
\end{theorem}

We will need the following lemma.
\begin{lemma}\label{estimates}
Let $f_1,g_1,e_1,f_2,g_2,e_2 \in C^r$ be such that
$$f_1(x,h)=g_1(x,h)+h^{r+1}e_1(x,h),$$
$$f_2(x,h)=g_2(x,h)+h^{r+1}e_2(x,h).$$
Then, there exists functions $e_{12}$ and $\bar{e}_1$ bounded on compact sets such that
$$f_2(f_1(x,h),h)=g_2(g_1(x,h),h)+h^{r+1}e_{12}(g_1(x,h),h),$$
$$f_1^{-1}(y(h))=g_1^{-1}(y(h))+h^{r+1}\bar{e}_1(y(h)).$$
\end{lemma}
\begin{proof}
\begin{align*}
f_2(f_1(x,h),h)&=f_2(g_1(x,h)+h^{r+1}e_1(x,h),h)\\&=g_2(g_1(x,h)+h^{r+1}e_1(x,h),h)+h^{r+1}e_2(g_1(x,h)+h^{r+1}e_1(x,h),h)\\&=g_2(g_1(x,h),h)+h^{r+1}\tilde{e}_1(g_1(x,h),h)+h^{r+1}e_2(g_1(x,h)+h^{r+1}e_1(x,h),h),
\end{align*}
where $\tilde{e}_1$ is bounded on compact set. This last line comes from combining compactness of the set with the smoothness of the functions to obtain a Lipschitz property of the form, 
$$\|g_2(g_1(x,h)+h^{r+1}e_1(x,h),h)-g_2(g_1(x,h),h)\| \leq Ch^{r+1}.$$
For each choice of $(x,h)$, equality holds for a particular choice of constant, which defines $\tilde{e}_1$ and establishes its smoothness as a function. Adding $e_2$ to $\tilde{e}_1$ we obtain a function $e_{12}$, which is also bounded on compact sets such that,
$$ f_2(f_1(x,h),h)=g_2(g_1(x,h),h)+h^{r+1}e_{12}(g_1(x,h),h).$$
Let $y=f_1(x,h)$, and note that by definition,
$$f_1^{-1}(f_1(x,h))=g_1^{-1}(g_1(x,h)).$$
Since $g_1^{-1}(y)=g_1^{-1}(g_1(x,h)+h^{r+1}e_1(x,h))$, then
$$\|g_1^{-1}(y)-f_1^{-1}(y)\|=\|g_1^{-1}(y)-g_1^{-1}(g_1(x,h))\| \leq \bar{C} h^{r+1}.$$
From this, it follows that there exists a function $\bar{e}_1$ bounded on compact sets such that,
$$f_1^{-1}(y)=g_1^{-1}(y)+h^{r+1}\bar{e}_1(y).$$
\end{proof}

Now we are ready for the proof of the theorem.
\begin{proof}
By assumption there is some bounded continuously differentiable function $e$ such that,
$$H^+_d(q(0),p(h),h)=H^{+,E}_d(q(0),p(h),h)+h^{r+1}e(q(0),p(h),h).$$
Differentiating yields,
$$D_1H^+_d(q(0),p(h),h)=D_1H^{+,E}_d(q(0),p(h),h)+h^{r+1}D_1e(q(0),p(h),h),$$
where $\|D_1e(q(0),p(h),h)\| \leq \tilde{C}$. This implies,
$$\| \mathbb{F}^-H_d^+(q(0),p(h),h) - \mathbb{F}^-H^{+,E}_d(q(0),p(h),h)\| \leq \tilde{C}h^{r+1}.$$
Now combining this with the fact that $\tilde{F}_{H^+_d}=\mathbb{F}^+H^+_d \circ (\mathbb{F}^-H^+_d)^{-1}$ and applying Lemma \ref{estimates}, we have,
$$\tilde{F}^h_{H^+_d}=\tilde{F}^h_{H_d^{+,E}} + \mathcal{O}(h^{r+1}).$$
\end{proof}

Determining the order of a variational integrator is greatly simplified via the above theorems, which relate the order of the integrator to the order to which the associated discrete Lagrangian or discrete right Hamiltonian approximates the corresponding exact generating function. Similarly, it was shown in \cite{MaWe2001} that one can determine whether or not the variational integrator is a symmetric method by examining the corresponding discrete Lagrangian. We would like to extend this result to the case of discrete Hamiltonians.

\subsection{Symmetric Methods}
\begin{definition}[see Chapters II.3 and V of \cite{HaLuWa2006}] A numerical one-step method $\Phi_h$ is called \textbf{symmetric} or \textbf{time-reversible}, if it satisfies
 $$\Phi_h \circ \Phi_{-h} = id$$
 or equivalently
 $$\Phi_h=\Phi^{-1}_{-h}.$$
 The \textbf{adjoint} of a numerical one-step method, denoted $\Phi_h^*$, is defined as
 $$\Phi_h^*=\Phi^{-1}_{-h}.$$
\end{definition}
A numerical one-step method is a symmetric method if it is self-adjoint(i.e. $\Phi_h=\Phi_h^*$). The adjoint of a discrete Lagrangian, $L_d^*$, is defined as
$$L_d^*(q_0,q_1,h)=-L_d(q_1,q_0,-h).$$
The discrete Lagrangian is called self-adjoint if $L_d^*(q_0,q_1,h)=L_d(q_0,q_1,h)$. The following theorem from \cite{MaWe2001} relates the self-adjointness of the discrete Lagrangian with the self-adjointness of the corresponding variational integrator.
\begin{theorem}[Theorem 2.4.1 of \cite{MaWe2001}]
The discrete Lagrangian (or an equivalent discrete Lagrangian), $L_d$, is self-adjoint if and only if the method associated to the corresponding discrete Hamiltonian map is self-adjoint (i.e. symmetric).
\end{theorem}
In many cases it is easier to check if the discrete Lagrangian is self-adjoint, rather than checking the variational integrator itself. We seek a definition for the adjoint of a discrete right Hamiltonian.

 The adjoint of a one-step method $(q_1,p_1)=\Phi_h(q_0,p_0)$ can be obtained by reversing the direction of time, and reversing the roles of the initial data and terminal solution, i.e., $(q_0,p_0)=\Phi^*_{-h}(q_1,p_1)$. This corresponds to swapping out $(q_0,p_0,q_1,p_1,h)$ for $(q_1,p_1,q_0,p_0,-h)$. This motivates the definition of the adjoint of a Type II generating function as a Type III generating function and vice versa. In particular, given a Type II discrete Hamiltonian $H_d^+$, we seek a definition for the Type III adjoint $(H_d^+)^*$ that will satisfy $F^h_{(H_d^+)^*} = (F^h_{H_d^+})^*$. Let $F^h_{(H_d^+)^*}(q_0,p_0)=(q_1,p_1)$. Then, we want
 \begin{align*}
 (q_1,p_1) &=F^h_{(H_d^+)^*}(q_0,p_0)\\
 &=(F^h_{H_d^+})^*(q_0,p_0) \\
 &=(F^{-h}_{H_d^+})^{-1}(q_0,p_0).
 \end{align*}
 This implies $F^{-h}_{H_d^+}(q_1,p_1) = (q_0,p_0)$, which together with  $F^h_{(H_d^+)^*}(q_0,p_0)=(q_1,p_1)$ yield the respective sets of equations,
\begin{align*}
p_1 &= D_1H_d(q_1,p_0;-h), \\
q_0 &= D_2H_d(q_1,p_o;-h), 
\end{align*}
and
\begin{align*}
p_1 &= -D_2(H_d)^*(p_0,q_1;h), \\
q_0 &= -D_1(H_d)^*(p_0,q_1;h).
\end{align*}
Comparing these equations we see that setting $(H_d^+)^*(p_0,q_1;h) = -H_d^+(q_1,p_0;-h)$ satisfies $F^h_{(H_d^+)^*} = (F^h_{H_d^+})^*$. A similar calculation yields an analogous expression for the adjoint of a Type III generating function $H_d^-$.
\begin{definition}
Given a Type II/III generating function, $H_d^\pm$, define the \textbf{adjoint} as the Type III/II generating function, $(H_d^\pm)^*$, where $F_{(H_d^\pm)^*}^h(q_0,p_0) = (q_1,p_1)$, as
\begin{equation}
(H_d^+)^*(p_0,q_1;h) = -H_d^+(q_1,p_0;-h),
\end{equation}
\begin{equation}
(H_d^-)^*(q_0,p_1;h) = -H_d^-(p_1,q_0;-h).
\end{equation}
\end{definition}

\begin{example}
The symplectic Euler-A method for a Lagrangian of the form $L(q,\dot{q}) = \frac{1}{2}\dot{q}^TM\dot{q} - V(q)$ is given by,
\begin{align*}
p_1 &= p_0 - h\nabla V(q_0), \\
 q_1 &= q_0 + hM^{-1}p_1.
\end{align*}
The corresponding discrete right Hamiltonian is given by 
\begin{align*}
H_d^+(q_0,p_1,h) &= p_1^T(q_0 + hM^{-1}p_1) - h[p_1^TM^{-1}p_1 - H(q_0,p_1)], \\
&= p_1^Tq_0 +h H(q_0,p_1).
\end{align*}
The adjoint of this method is given by symplectic Euler-B,
\begin{align*}
q_1 = q_0 + hM^{-1}p_0, \\
p_1 = p_0 - h\nabla V(q_1).
\end{align*}
We now derive the corresponding adjoint of the discrete right Hamiltonian for symplectic Euler-A.
\begin{align*}
(H_d^+)^*(q_1,p_0;h) &= -H^+_d(p_0,q_1;-h) \\
&= -p_0^T(q_1 - hM^{-1}p_0) - h[p_0^TM^{-1}p_0 - H(q_1,p_0)] \\
&= -p_0^Tq_1 + h H(q_1,p_0).
\end{align*}
We can verify that this generates symplectic Euler-B by applying the discrete left Hamilton's equations,
\begin{align*}
q_0 &= -D_1(H^+_d)^*(p_0,q_1;h) \\
&= D_2H^+_d(q_1,p_0;-h) \\
&= q_1 - hM^{-1}p_0, \\
p_1 &= -D_2(H^+_d)^*(p_0,q_1;h) \\
&= D_1H^+_d(q_1,p_0;-h) \\
&= p_0 - h\nabla V(q_1).
\end{align*}
Solving the first equation for $q_1$ gives symplectic Euler-B, as expected.
\end{example}

\begin{theorem}
$(H_d^\pm)^{**}=H_d^\pm$.
\end{theorem}
\begin{proof}
We consider the case of the Type II generating function $H_d^+$. Let $F^h_{(H_d^+)^{**}}(q_0,p_0)=(q_1,p_1)$. Since $(H_d^+)^*$ is a Type III generating function, applying the definition of the adjoint twice gives
\begin{align*}
(H_d^+)^{**}(q_0,p_1;h) &= -(H_d^+)^*(p_1,q_0;-h) \\
&= H_d^+(q_0,p_1;h),
\end{align*}
and a similar calculation shows that this holds for the Type III generating function $H_d^-$ as well.
\end{proof}

Since the notion of the adjoint that we introduced converts a Type II to a Type III generating function, for a discrete Hamiltonian to be self-adjoint, we need to compare the adjoint to the Legendre transformation of the discrete Hamiltonian, which is given by,
\[H_d^-(p_k, q_{k+1})=-p_k q_k - p_{k+1} q_{k+1}+H_d^+(q_k, p_{k+1}),\]
where we view $p_{k+1}$ and $q_k$ as functions of $p_k$ and $q_{k+1}$. Then, the following calculation shows that these two generating functions generate the same symplectic map, i.e., $F_{H_d^-}=F_{H_d^+}$,
\begin{align*}
-D_1 H_d^- (p_k, q_{k+1})
&=q_k +p_k\frac{\partial q_k}{\partial p_k}+\frac{\partial p_{k+1}}{\partial p_k} q_{k+1}-D_1 H_d^+(q_k, p_{k+1})\frac{\partial q_k}{\partial p_k}-D_2 H_d^+(q_k, p_{k+1})\frac{\partial p_{k+1}}{\partial p_k}\\
&=q_k+\left(p_k-D_1 H_d^+(q_k, p_{k+1})\right)\frac{\partial q_k}{\partial p_k}+\left(q_{k+1}-D_2 H_d^+(q_k, p_{k+1}\right)\frac{\partial p_{k+1}}{\partial p_k},\\
-D_2 H_d^-(p_k, q_{k+1})&= p_k\frac{\partial q_k}{\partial q_{k+1}}+\frac{\partial p_{k+1}}{\partial q_{k+1}}q_{k+1}+p_{k+1}-D_1 H_d^+(q_k, p_{k+1})\frac{\partial q_k}{\partial q_{k+1}}-D_2 H_d^+(q_k, p_{k+1})\frac{\partial p_{k+1}}{\partial q_{k+1}}\\
&=p_{k+1}+\left(p_k-D_1 H_d^+(q_k, p_{k+1})\right)\frac{\partial q_k}{\partial q_{k+1}}+\left(q_{k+1}-D_2 H_d^+(q_k, p_{k+1})\right)\frac{\partial p_{k+1}}{\partial q_{k+1}}.
\end{align*}

\begin{definition}
A Type II/III generating function is \textbf{self-adjoint}, if it is equal (up to equivalency) to the Legendre transform of its adjoint.
\end{definition}
Note that this definition implies that a discrete right Hamiltonian is self-adjoint if its adjoint is equal (up to equivalency) to the associated discrete left Hamiltonian, i.e.,  $(H_d^+)^* = H_d^-$.
\begin{corollary}
Given a self-adjoint discrete right Hamiltonian, i.e., $H_d^-=(H_d^+)^*$, the method associated to the discrete right Hamiltonian map is self-adjoint. Likewise, if a method coming from a discrete right Hamiltonian map is self-adjoint, then the associated discrete right Hamiltonian is self-adjoint.
\end{corollary}
\begin{proof}
Assume $H_d^-=(H_d^+)^*$. Then,
$$(F_{H_d^+})^*=F_{(H_d^+)^*}=F_{H_d^-}=F_{H_d^+},$$
and so, by definition, the map is self-adjoint.
Now assume $F_{H_d^+}=(F_{H_d^+})^*$. Then,
$$F_{H_d^-}=F_{H_d^+}=(F_{H_d^+})^*=F_{(H_d^+)^*},$$
which implies $(H_d^+)^*=H_d^-$  (up to equivalency) and, by definition, the discrete right Hamiltonian is self-adjoint.
\end{proof}

The previous corollary allows for an easy way to check if a variational integrator is self-adjoint. Assuming the Hamiltonian flow is time-reversible, it follows that the exact discrete right Hamiltonian is self-adjoint. This can also be shown using the definition of a self-adjoint exact discrete right Hamiltonian.

\begin{corollary}
The exact  discrete right Hamiltonian, $H_d^{+,E}$, is self-adjoint.
\end{corollary}
\begin{proof} A direct calculation shows that
\begin{align*}
(H_d^{+,E})^*(p_0,q_1;h) &= -H_d^{+,E}(q_1,p_0;-h) \\
&= -\left(\tilde{p}(-h)^T\tilde{q}(-h) - \int_0^{-h} [\tilde{p}(\tau)^T\tilde{q}(\tau) - H(\tilde{q}(\tau),\tilde{p}(\tau))] d\tau\right) \\
&= -p(-h+h)^Tq(-h+h) - \int_{-h}^0 [p(\tau+h)^Tq(\tau+h) - H(q(\tau+h),p(\tau+h))] d\tau \\
&= -p(0)^Tq(0) - \int_0^h [p(t)^Tq(t) - H(q(t),p(t))] dt \\
&= H_d^{-,E}(p_0,q_1;h),
\end{align*}
where we used the fact that the time-reversed solution $(\tilde{q}(\tau),\tilde{p}(\tau))$ over the time domain $[-h,0]$ with $(q_1,p_0)$ boundary data is related to the solution curve $(q(t),p(t))$ over the time domain $[0,h]$ with $(q_0,p_1)$ boundary data by $(\tilde{q}(\tau),\tilde{p}(\tau))=(q(\tau+h),p(\tau+h))$.
\end{proof}

The definition of the adjoint also provides a simple way to construct symmetric methods. Given any method defined by $H_d$, we can construct a symmetric method using composition, for example, $F^{\frac{h}{2}}_{H_d} \circ F^{\frac{h}{2}}_{H_d^*}$, which is nothing more than composing a half-step of the adjoint method with a half-step of the method. It is well-known that this leads to a symmetric method, as the following calculation demonstrates,
\begin{align*}
(F^{\frac{h}{2}}_{H_d} \circ F^{\frac{h}{2}}_{H_d^*})^* & = (F^{\frac{h}{2}}_{H_d^*})^* \circ (F^{\frac{h}{2}}_{H_d})^* \\
&= F^{\frac{h}{2}}_{H_d^{**}} \circ F^{\frac{h}{2}}_{H_d^*} \\
&= F^{\frac{h}{2}}_{H_d} \circ F^{\frac{h}{2}}_{H_d^*}.
\end{align*}
More generally, a composition method of the form,
\[ F^{\alpha_s h}_{H_d} \circ F^{\beta_s h}_{H_d^*} \circ \cdots \circ F^{\beta_2 h}_{H_d^*} \circ F^{\alpha_1 h}_{H_d} \circ F^{\beta_1 h}_{H_d^*},\]
where $\alpha_{s+1-i}=\beta_i$ for $i=1,\ldots,s$, will be symmetric. For a more in depth discussion of symmetric composition methods, see Chapter V.3 of \cite{HaLuWa2006}.

\section{Discrete Lagrangians versus Discrete Hamiltonians}

A symplectic method defines a symplectic map, and for any symplectic map there exists, locally, a generating function in terms of at least one of the pairs, $(q_0,q_1)$, $(q_0,p_1)$, $(q_1,p_0)$, which corresponds to a Type I, Type II, and Type III generating function, respectively. Given the respective pair forms an independent set of coordinates, then we are guaranteed the existence, locally, of the corresponding generating function. Therefore, it is not a very interesting question to ask if there is a discrete Hamiltonian or discrete Lagrangian associated with a particular symplectic method.

There are two general methods of constructing a variational integrator, the shooting-based method introduced in \cite{LeSh2011_sbvi} and the Galerkin variational integrators introduced in \cite{MaWe2001} and analyzed in \cite{HaLe2015}. In particular, shooting-based variational integrators are constructed from a choice of a numerical quadrature scheme and an underlying one-step method, whereas Galerkin variational integrators are constructed from the choice of a numerical quadrature scheme and a finite-dimensional function space. With this in mind, an interesting question to ask is the following: If we are given a discrete Lagrangian or Hamiltonian constructed using the shooting-based or Galerkin approach with a particular choice of quadrature rule and either underlying one-step method or finite-dimensional function space, will constructing a different type of generating function based on the same approximation scheme lead to an equivalent symplectic method?
\\

It was shown in \cite{LeZh2011} that the Galerkin variational integrator construction leads to equivalent discrete Lagrangian and discrete Hamiltonian methods for the same choice of quadrature rule and finite-dimensional function space, and the result is given in the following theorem.
\begin{theorem}[Proposition 4.1 of \cite{LeZh2011}]
If the continuous Hamiltonian $H(q,p)$ is hyperregular and we construct a Lagrangian $L(q,\dot{q})$ by the Legendre transformation, then the generalized Galerkin Hamiltonian variational integrator (see \cite{LeZh2011}) and the generalized Galerkin Lagrangian variational integrator, associated with the same choice of basis functions and numerical quadrature formula, are equivalent.
\end{theorem}
Does this hold for other types of variational integrators? To begin to address this question we will examine the approximation scheme of a Taylor variational integrator, which is a variant of the shooting-based variational integrator of \cite{LeSh2011_sbvi}, and is also related to the prolongation--collocation variational integrators developed in \cite{LeSh2011}.

\subsection{Taylor Variational Integrators}
Consider the exact discrete Lagrangian, which is defined as,
\begin{align*}
L_d^E(q_0,q_1;h)&=\int_0^h L(q_{01}(t),\dot q_{01}(t)) dt,
\end{align*}
where $q_{01}(0)=q_0,$ $q_{01}(h)=q_1,$ and $q_{01}$ satisfies the Euler--Lagrange equation in the time interval $(0,h)$. Then, the Taylor discrete Lagrangian is constructed as follows:
\begin{enumerate}
\item Construct a $(r+1)$-order Taylor expansion on the configuration manifold about the initial time and implicitly solve for an approximation to the initial velocity $\tilde{v}_0$,
$$q_1 = \pi_Q \circ \Psi^{(r+1)}_h(q_0,\tilde{v}_0).$$
\item Pick a quadrature rule of order $s$ with quadrature weights and nodes given by $(b_i,c_i)$ for $i=1,\ldots, m$.
\item Construct an $r$-order Taylor method on the tangent bundle, $TQ$, and use it to generate approximations of $(q(t),v(t))$ at the quadrature nodes,
$$(q_{c_i},v_{c_i})= \Psi^{(r)}_{c_ih}(q_0,\tilde{v}_0).$$
\item Apply the quadrature rule to form the discrete Lagrangian of order $\min(r+1,s)$,
$$L_d(q_0,q_1;h)= h\sum_{i=1}^m b_{i} L\left(\Psi_{c_ih}^{(r)}(q_0,\tilde{v}_0)\right).$$
\end{enumerate}
Then, the Taylor variational integrator is implicitly defined by the implicit discrete Euler--Lagrange equations,
\begin{equation}
p_0=-D_1 L_d(q_0, q_1),\qquad p_1=D_2 L_d(q_0, q_1).\label{IDEL}
\end{equation}
\begin{example}As an example consider a first-order Taylor discrete Lagrangian. 
\begin{enumerate}
\item Solve $q_1 = q_0 + h\tilde{v}_0$ for $\tilde{v}_0$. This implies $\tilde{v}_0=\frac{q_1-q_0}{h}$.
\item The quadrature rule used here will be the rectangular rule about the initial point with weight and node $(1,0)$.
\item The zeroth-order Taylor expansion trivially yields,
$$(q_0,\tilde{v}_0)= \Psi^{1}_0(q_0,\tilde{v}_0).$$
\item Using the quadrature rule, we have the discrete Lagrangian,
$$L_d(q_0,q_1;h)= hL\Big(q_0,\frac{q_1-q_0}{h}\Big).$$
\end{enumerate}
Assuming a Lagrangian of the form $L(q,\dot{q})=\frac{1}{2}\dot{q}^TM\dot{q}-V(q)$, the implicit discrete Euler--Lagrange equations \eqref{IDEL} yield
\begin{align*}
p_0 = M\frac{q_1-q_0}{h} - h\nabla V(q_0), \qquad p_1=M\frac{q_1-q_0}{h}.
\end{align*}
Rearranging these equations, we see that this corresponds to symplectic Euler-A.
\end{example}

The boundary-value formulation of the exact discrete right Hamiltonian is given by,
\begin{align*}
H_d^{+,E}(q_0,p_1)& =  \left(p_1 q_1 - \int_0^T \left[ p \dot{q}-H(q, p) \right]
dt\right),
\end{align*}
where $(q(t),p(t))$ satisfy Hamilton's equations with boundary conditions $q(0)=q_0$, $p(T)=p_1$.
Now let us consider the construction of a Taylor discrete right Hamiltonian.
\begin{enumerate}
\item Construct a $r$-order Taylor expansion on the cotangent bundle, $T^*Q$, and solve for $\tilde{p}_0$,
$$p_1 = \pi_{T^*Q} \circ \Psi^{(r)}_h(q_0,\tilde{p}_0).$$
\item Pick a quadrature rule of order $s$ with quadrature weights and nodes given by $(b_i,c_i)$ for $i=1,\ldots,m$.
\item Use a $r$-order Taylor method to generate approximations of $(q(t),p(t))$ at the quadrature nodes,
$$(q_{c_i},p_{c_i})= \Psi^{(r)}_{c_ih}(q_0,\tilde{p}_0),$$
and use a $(r+1)$-order Taylor method on the configuration manifold to generate the approximation to the boundary term $q_1$,
$$\tilde{q}_1=\pi_Q \circ \Psi^{(r+1)}_h(q_0,\tilde{p}_0).$$
\item Use the quadrature rule and approximate boundary term, $\tilde{q}_1$, to construct the discrete right Hamiltonian of order $\min(r+1,s)$,
$$H_d^+(q_0,p_1;h)=p_1^T\tilde{q}_1 - h \sum^m_{i=1} \Big[p_{c_i}^T\dot{q}_{c_i} - H\Big(\Psi^{(r)}_{c_ih}(q_0,\tilde{p}_0)\Big)\Big],$$
where $\dot{q}_{c_i}$ is obtained by inverting the continuous Legendre transform, $(q_{c_i},p_{c_i})=\mathbb{F}L(q_{c_i},\dot{q}_{c_i})$.
\end{enumerate}
The method is implicitly defined by the implicit discrete Hamilton's equations,
\begin{equation}
q_1 = D_2H_d^+(q_0, p_1), \qquad p_0 = D_1 H_d^+(q_0, p_1). \label{IHE}
\end{equation}

\begin{example}We now construct a first-order Taylor discrete right Hamiltonian using the rectangular rule about the initial point.
\begin{enumerate}
\item The zeroth-order Taylor expansion yields $p_1=\tilde{p}_0$.
\item The rectangular rule about the initial point is given by weight and node $(1,0)$.
\item The boundary term, $\tilde{q}_1$, is given by the first-order Taylor method, $\tilde{q}_1=q_0+hM^{-1}p_1$.
\item The discrete right Hamiltonian is given by,
$$H_d^+(q_0,p_1;h)=p_1^T(q_0+hM^{-1}p_1)-h(p_1^TD_2H(q_0,p_1)-H(q_0,p_1)).$$
\end{enumerate}
Assuming a Hamiltonian of the form $H(q,p)=\frac{1}{2}p^TM^{-1}p+V(q)$, the implicit discrete Hamilton's equations \eqref{IHE} yield
\begin{align*}
q_1 = q_0 + hM^{-1}p_1, \qquad p_0= p_1 + h\nabla V(q_0),
\end{align*}
which when rearranged recovers symplectic Euler-A. 
\end{example}
In this case the discrete Lagrangian and discrete right Hamiltonian constructed via the Taylor variational integrator method have given rise to the same method. However, had we chosen to apply the rectangular rule about the end point, then the resulting Taylor discrete Lagrangian method would be symplectic Euler-B, but the  Taylor discrete \textbf{right} Hamiltonian method would not be symplectic Euler-B. Instead, had we constructed a Taylor discrete \textbf{left} Hamiltonian using the rectangular quadrature rule about the end point, then the resulting method would be symplectic Euler-B. To understand why, all we need to do is look at the independent coordinates for each respective generating function. The discrete Lagrangian is defined in terms of $(q_0,q_1)$, which means that the rectangular rule around either the end point or the initial point will imply the nonlinear term, $V(q)$, involves the true respective value and will not be implicit for low order expansions. The discrete right Hamiltonian is defined in terms of $(q_0,p_1)$, so the rectangular rule around the initial point will involve $V(q_0)$, but applying the rule about the endpoint will involve $V(q_0+hp_1)$. On the other hand it is the exact opposite for the discrete left Hamiltonian, which is defined in terms of $(q_1,p_0)$. The following tables summarize these statements.
\\

\begin{center}
    \begin{tabular}{|| p{3cm} | p{3cm} | p{3cm} | p{3cm} ||}
    \hline
    Quad. Rule & & & \\ (Initial Point) & Type I ($q_0,q_1$) & Type II ($q_0,p_1$) & Type III ($q_1,p_0$) \\ \hline
    \multirow{4}{*}{Approx.} & $q_0=q_0$ & $q_0=q_0$ & $q_0=q_1-hM^{-1}p_0$ \\ & $q_1=q_1$ & $q_1=q_0+hM^{-1}p_1$ & $q_1=q_1$ \\ & $v_0=\frac{q_1-q_0}{h}$ & $p_0=p_1$ & $p_0=p_0$ \\ & $v_1=\frac{q_1-q_0}{h}$ & $p_1=p_1$ & $p_1=p_0$ \\ \hline
    \multirow{3}{*}{Transforms} & $p_0 = M(\frac{q_1-q_0}{h}) + h\nabla V(q_0)$ & $p_0 = p_1 + h\nabla V(q_0)$ & 
    $q_0 = q_1 - hM^{-1}p_0$ \\ 
    & & & $\quad+ h^2M^{-1}\nabla V(q_1-hM^{-1}p_0)$\\
    & $p_1 = M(\frac{q_1-q_0}{h})$ & $q_1 = q_0 + hM^{-1}p_1$ & $p_1 = p_0 - h\nabla V(q_1-hM^{-1}p_0)$ \\ \hline
    \multirow{2}{*}{Method} & $q_1 = q_0 +hM^{-1}p_1$ & $q_1 = q_0 + hM^{-1}p_1$ & $q_1 = q_0 +hM^{-1}p_1$ \\ & $p_1 = p_0 - h\nabla V(q_0)$ & $p_1 = p_0 - h\nabla V(q_0)$ & $p_1 = p_0 - h\nabla V(q_1-hM^{-1}p_0)$ \\ \hline
    Same as & & & \\ Type I Method & NA & \hl{Yes} & \hl{No} \\ \hline
    Approx. satisfies & & & \\ $-D_1L_d(q_0,q_1)=p_0$ & NA & No & No \\ \hline
    Approx. satisfies & & & \\ $D_2L_d(q_0,q_1)=p_1$ & NA & Yes & Yes \\ \hline
    Independent Variable & & &\\ satisfies Legendre & & &\\ Transform & NA & \hl{Yes} & \hl{No} \\ \hline
    \end{tabular}
\end{center}


\begin{center}
    \begin{tabular}{|| p{3cm} | p{3cm} | p{3cm} | p{3cm} ||}
    \hline
    Quad. Rule & & &\\ (End Point) & Type I ($q_0,q_1$) & Type II ($q_0,p_1$) & Type III ($q_1,p_0$) \\ \hline
    \multirow{4}{*}{Approx.} & $q_0=q_0$ & $q_0=q_0$ & $q_0=q_1-hM^{-1}p_0$ \\ & $q_1=q_1$ & $q_1=q_0+hM^{-1}p_1$ & $q_1=q_1$ \\ & $v_0=\frac{q_1-q_0}{h}$ & $p_0=p_1$ & $p_0=p_0$ \\ & $v_1=\frac{q_1-q_0}{h}$ & $p_1=p_1$ & $p_1=p_0$ \\ \hline
    \multirow{3}{*}{Transforms} & $p_0 = M(\frac{q_1-q_0}{h})$ & $p_0 = p_1 + h\nabla V(q_0+hM^{-1}p_1)$ & $q_0 = q_1 - hM^{-1}p_0$ \\ & $p_1 = M(\frac{q_1-q_0}{h})-h\nabla V(q_1)$ & $q_1 = q_0 + hM^{-1}p_1 $ & $p_1 = p_0 - h\nabla V(q_1)$ \\ 
    & & $\quad+ h^2M^{-1}\nabla V(q_0+hM^{-1}p_1)$ & \\\hline
    \multirow{2}{*}{Method} & $q_1 = q_0 +hM^{-1}p_0$ & $q_1 = q_0 + hM^{-1}p_0$ & $q_1 = q_0 +hM^{-1}p_0$ \\ & $p_1 = p_0 - h\nabla V(q_1)$ & $p_1 = p_0 - h\nabla V(q_0+hM^{-1}p_1)$ & $p_1 = p_0 - h\nabla V(q_1)$ \\ \hline
    Same as & & &\\ Type I Method & NA & \hl{No} & \hl{Yes} \\ \hline
    Approx. satisfies & & &\\ $-D_1L_d(q_0,q_1)=p_0$ & NA & Yes & Yes \\ \hline
    Approx. satisfies & & &\\ $D_2L_d(q_0,q_1)=p_1$ & NA & No & No \\ \hline
    Independent Variable & & & \\ satisfies Legendre & & &\\ Transform & NA & \hl{No} & \hl{Yes} \\ \hline
    \end{tabular}
\end{center}


Therefore, the answer to our original question is that in general, a fixed approximation scheme used to construct a discrete Lagrangian will not generate the same method when it is used to construct a discrete Hamiltonian. It seems that if the approximated value of $q_1$ or $q_0$ (for Type II and Type III, respectively) satisfies the appropriate discrete Legendre transform, then the Type II or Type II approximation will yield the same method as the Type I approximation. In general, how might the two resulting methods vary? A complete characterization of this issue is subtle, and beyond the scope of this paper, but it will be a topic of future work. For now, we will consider how the two approaches differ when combined with the method of averaging, which will also serve to illustrate how the type of boundary data can affect the numerical properties of the method.

\subsection{Averaged Hamiltonians}
Averaging methods have played a role in solving differential equations since at least as far back as the time of Lagrange (see \cite{Ver2000}), and they continue to play a key role particularly in the field of numerical differential equations applied to nearly integrable systems or problems with multiple timescales.
We consider perturbed Hamiltonian systems with Hamiltonians of the form,
\begin{equation}
H = H^{(A)} + \epsilon H^{(B)}, \label{pham}
\end{equation}
where $\epsilon \ll 1$ and the dynamics of the Hamiltonian system corresponding to $H^{(A)}$ is exactly solvable or at the very least cheap to approximate. We call this an almost-integrable system. The motivation being that the dynamics of the system are largely influenced by an integrable Hamiltonian with simpler dynamics, but smaller influences also play a role in the overall dynamics. An example is the classic $n$-body problem of the solar system, where a particular planet's trajectory is largely influenced by the sun, but other planets and nearby objects also play a role. Averaging methods can be constructed to exploit the larger influence of $H^{(A)}$ on the dynamics of the system by averaging out the smaller influences. Ideally, averaging techniques will allow for larger time steps to be used while still yielding a reasonable approximation to the solution.
\\

A variational integrator for such a system was proposed in \cite{Farr2009} using a discrete Lagrangian formulation, which drew inspiration from the kick-drift-kick leapfrog method (see \cite{WiHo1991}). We will discuss the Lagrangian formulation (hereafter referred to as the averaged Lagrangian) and in addition construct an analogous method in terms of a discrete right Hamiltonian (referred to as the averaged Hamiltonian). The Lagrangian corresponding to \eqref{pham} is given by,
\begin{equation}
L = L^{(A)} + \epsilon L^{(B)}.
\end{equation}
Making the assumption that $L^{(B)}(q(t),\dot{q}(t))=-V^{(B)}(q(t))$, then the kick-drift-kick leapfrog method is given by the discrete Lagrangian,
\begin{align*}
L_d(q_0,q_1;h) = L_d^{(A),E}(q_0,q_1;h) - \epsilon \frac{h}{2} \Bigl[ V^{(B)}(q_0) + V^{(B)}(q_1) \Bigr],
\end{align*}
where the trapezoid quadrature rule has been used to approximate $\int_0^h V^{(B)}(q(t)) dt$.
The discrete Hamiltonian map is implicitly defined by,
\begin{align*}
-p_0 &= D_1 L_d^{(A),E}(q_0,q_1;h) - \epsilon \frac{h}{2} \nabla V^{(B)}(q_0), \\
p_1 &= D_2 L_d^{(A),E}(q_0,q_1;h) - \epsilon \frac{h}{2} \nabla V^{(B)}(q_1).
\end{align*}
Rearranging terms we have,
\begin{align*}
-(p_0- \epsilon \frac{h}{2} \nabla V^{(B)}(q_0) ) &= D_1 L_d^{(A),E}(q_0,q_1;h), \\
p_1 &= D_2 L_d^{(A),E}(q_0,q_1;h) - \epsilon \frac{h}{2} \nabla V^{(B)}(q_1).
\end{align*}
This can be interpreted as first kicking $p_0$ by $- \epsilon \frac{h}{2} \nabla V^{(B)}(q_0)$, then we drift by $L_d^{(A),E}$ to get $q_1$, and finally we kick $p_1^{(A)}$ by $- \epsilon \frac{h}{2} \nabla V^{(B)}(q_1)$ to get $p_1$. This method has local truncation error of size $\mathcal{O}(\epsilon h^3)$. \\

The method of interest, proposed by Will Farr, improves the local truncation error to $\mathcal{O}(\epsilon^2 h^3)$, and is defined in terms of a discrete Lagrangian, $L_d$. We will reproduce the construction of the discrete Lagrangian formulation, then introduce a discrete right Hamiltonian formulation, $H_d^+$, in the same spirit. To clarify notation we will be assuming that $(q_0,p_0)$ are the initial conditions for both implementations, and we introduce $(q_{1,L_d},p_{1,L_d})$ and $(q_{1,H_d^+},p_{1,H_d^+})$ to denote the respective numerical approximations after one timestep. The method proposed in \cite{Farr2009} used a discrete Lagrangian of the form,
\begin{align*}
L_d(q_0,q_{1,L_d},h) &= L_d^{(A),E}(q_0,q_{1,L_d};h) + \epsilon \int_0^h L^{(B)}(q_A(q_0,q_{1,L_d},t),\dot{q}_A(q_0,q_{1,L_d},t)) dt \\
&= L_d^{(A),E}(q_0,q_{1,L_d};h) - \epsilon \int_0^h V^{(B)}(q_A(q_0,q_{1,L_d},t)) dt,
\end{align*}
where we denote the trajectory corresponding to $L^{(A)}$ with boundary conditions $(q_0,q_1)$ by $(q_A(t),\dot{q}_A(t))$. The idea is to use the dynamics of $L^{(A)}$, which is either solved for exactly or efficiently approximated, to average the contribution of $L^{(B)}$ to the dynamics.
The corresponding discrete Hamiltonian map is given implicitly by
\begin{subequations}\label{p0}
\begin{align}
-p_0 &= D_1L_d^{(A),E}(q_0,q_{1,L_d};h) - \epsilon \int_0^h D_1V^{(B)}(q_A(q_0,q_{1,L_d},t)) dt,  \\
p_{1,L_d} &= D_2L_d^{(A),E}(q_0,q_{1,L_d};h) - \epsilon \int_0^h D_2V^{(B)}(q_A(q_0,q_{1,L_d},t)) dt.
\end{align}
\end{subequations}
The method defined by the above has local truncation error of size $\mathcal{O}(\epsilon^2 h^3)$. Using the notation $p_0^A(q_0,q_{1,L_d})=-D_1L_d^{(A),E}(q_0,q_{1,L_d};h)$ and $p_1^A(q_0,q_{1,L_d})=D_2L_d^{(A),E}(q_0,q_{1,L_d};h)$, we rearrange the above equations to get
\begin{subequations}\label{Leqtns}
\begin{gather}
p_0 - \epsilon \int_0^h D_1V^{(B)}(q_A(q_0,q_{1,L_d},t)) dt = p_0^A(q_0,q_{1,L_d}),\\
p_{1,L_d} = p_1^A(q_0,q_{1,L_d}) - \epsilon \int_0^h D_2V^{(B)}(q_A(q_0,q_{1,L_d},t)) dt. 
\end{gather}
\end{subequations}
This can be interpreted as first implicitly kicking $p_0$ by $- \epsilon \int_0^h D_1V^{(B)}(q_A(q_0,q_{1,L_d},t)) dt$, which is the impulse due to the force associated with potential $V^{(B)}$ averaged over the trajectory generated by $L^{(A)}$. Then by implicitly drifting along $L_d^{(A),E}$ we arrive at $q_{1,L_d}$, and finally kicking $p_1^A(q_0,q_{1,L_d})$ by the trajectory-averaged impulse $- \epsilon \int_0^h D_2V^{(B)}(q_A(q_0,q_{1,L_d},t)) dt$ to get $p_{1,L_d}$. In \cite{Farr2009}, it is noted that $- \epsilon \int_0^h D_1V^{(B)}(q_A(q_0,q_{1,L_d},t)) dt$ is an average along the trajectory generated by $L^{(A)}$ which, in general, gives more weight to the initial periods of the trajectory, while $- \epsilon \int_0^h D_2V^{(B)}(q_A(q_0,q_{1,L_d},t)) dt$ is an average along the trajectory generated by $L^{(A)}$ that, in general, favors the latter periods of the trajectory. The interpretation is not quite as clear as in the previous method due to the implicit nature of the equations, but nonetheless the role of averaging is quite apparent.

 Now let us consider the discrete right Hamiltonian given by the same form of approximation,
\begin{align*}
H_d^+(q_0,p_{1,H_d^+};h) = H_d^{(A),+,E}(q_0,p_{1,H_d^+};h) + \epsilon \int_0^h V^{(B)}(q_A(q_0,p_{1,H_d^+},t)) dt.
\end{align*}
The discrete right Hamiltonian map is given implicitly by
\begin{align*}
p_0 &= D_1 H_d^{(A),+,E}(q_0,p_{1,H_d^+};h) + \epsilon \int_0^h D_1 V^{(B)}(q_A(q_0,p_{1,H_d^+},t)) dt, \\
q_{1,H_d^+} &= D_2 H_d^{(A),+,E}(q_0,p_{1,H_d^+};h) + \epsilon \int_0^h D_2 V^{(B)}(q_A(q_0,p_{1,H_d^+},t)) dt. 
\end{align*}
Using the notation $p_0^A(q_0,p_{1,H_d^+}) = D_1 H_d^{(A),+,E}(q_0,p_{1,H_d^+};h)$ and $q_1^A(q_0,p_{1,H_d^+}) = D_2 H_d^{(A),+,E}(q_0,p_{1,H_d^+};h)$, we rearrange the equations to yield
\begin{subequations}\label{Heqtns}
\begin{gather}
p_0 - \epsilon \int_0^h D_1 V^{(B)}(q_A(q_0,p_{1,H_d^+},t)) dt = p_0^A(q_0,p_{1,H_d^+}),\label{Heqtns1}\\
q_{1,H_d^+} = q_1^A(q_0,p_{1,H_d^+}) + \epsilon \int_0^h D_2 V^{(B)}(q_A(q_0,p_{1,H_d^+},t)) dt.\label{Heqtns2} 
\end{gather}
\end{subequations}
This can be interpreted as first implicitly kicking $p_0$ by $- \epsilon \int_0^h D_1 V^{(B)}(q_A(q_0,p_{1,H_d^+},t)) dt$, then implicitly drifting by $H^{(A)}$ to get $p_{1,H_d^+}$. Finally, shifting $q_1^A(q_0,p_{1,H_d^+})$ by $\epsilon \int_0^h D_2 V^{(B)}(q_A(q_0,p_{1,H_d^+},t)) dt$ we arrive at $q_{1,H_d^+}$.
\begin{theorem}
The method defined implicitly by \eqref{Heqtns} has local truncation error $\mathcal{O}(\epsilon^2 h^3)$.
\end{theorem}
\begin{proof}
Using variational error analysis, we need to show
\begin{align*}
\mathcal{O}(\epsilon^2 h^3) &= H_d^{E,+} - H_d^+ \\
&= \Delta_A + \epsilon \Delta_B,
\end{align*}
where $\Delta_A$ is given by
\begin{align*}
p(h)^Tq(h) - \int_0^h [p(t)^T\dot{q}(t) - H^{(A)}(q(t),p(t))] dt - \Biggl(p_A(h)^Tq_A(h) - \int_0^h [p_A(t)^T\dot{q}_A(t) - H^{(A)}(q_A(t),p_A(t))] dt\Biggr), 
\end{align*}
and $\epsilon\Delta_B$ is given by
\begin{align*}
\epsilon \int_0^h \biggl[V^{(B)}(q(t)) - V^{(B)}(q_A(t))\biggr] dt.
\end{align*}
Using a functional Taylor expansion, $\Delta_A$ becomes
\begin{align*}
\Delta_A &= \frac{\delta}{\delta q_A} \left(\int_0^h [p_A(t)^T\dot{q}_A(t) - H^{(A)}(q_A(t),p_A(t))] dt \right) \delta q_A \\
&\qquad+ \frac{\delta^2}{\delta q_A^2} \left(\int_0^h [p_A(t)^T\dot{q}_A(t) - H^{(A)}(q_A(t),p_A(t))] dt \right) \delta q_A^2 + \mathcal{O}(\delta q_A^3),
\end{align*}
where $\delta q_A$ is the difference between $q$ and $q_A$. Noting that $q$ and $q_A$ differ in forces of order $\epsilon$ and $p$ differs from $p_A$ to first order in $\epsilon h$, implies that $\delta q_A$ is on the order of $\mathcal{O}(\epsilon h)$. This can be seen explicitly by comparing Taylor expansions about time zero. Since $q_A$ satisfies Hamilton's equations for $H^{(A)}$, the first variation vanishes (see Lemma 2.1 of \cite{LeZh2011}) leaving a term on the order of $h \delta q_A^2$. Therefore, we have
\begin{align*}
\Delta_A = \mathcal{O}(\epsilon^2 h^3).
\end{align*}
Likewise, a functional Taylor expansion for $\Delta_B$ yields,
\begin{align*}
\Delta_B = \frac{\delta}{\delta q_A}\left[ \int_0^h V^{(B)}(q_A(t))dt \right] \delta q_A + \mathcal{O}(\delta q_A^2).
\end{align*}
Noting that $V^{(B)}$ is only a function of $q_A$ and that $q$ differs from  $q_A$ on the order of $\epsilon h^2$, implies $ \epsilon \Delta_B = \mathcal{O}(\epsilon^2 h^3)$.
\end{proof}

Are the maps defined by $L_d$ and $H_d^+$ the same map? Or equivalently, is $H_d^+$ the Legendre transform of $L_d$? The answer is no, but to see this let us suppose it is true. The Legendre transform of $L_d(q_0,q_{1,L_d};h)$ is given by $p_{1,L_d}^Tq_{1,L_d} - L_d(q_0,q_{1,L_d};h)$, where $q_{1,L_d}$ is defined in terms of $p_{1,L_d}$ and $q_0$. Expanding this out we have,
\begin{align*}
p_{1,L_d}^Tq_{1,L_d} - L_d^{(A),E}(q_0,q_{1,L_d};h) + \epsilon \int_0^h V^{(B)}(q_A(q_0,q_{1,L_d},t)) dt.
\end{align*}
Given that $H_d^+(q_0,p_{1,H_d^+},h) = H_d^{(A),+,E}(q_0,p_{1,H_d^+},h) + \epsilon \int_0^h V^{(B)}(q_A(q_0,p_{1,H_d^+},t)) dt$, this implies that $V^{(B)}(q_A(q_0,q_{1,L_d},t)) = V^{(B)}(q_A(q_0,p_{1,H_d^+},t))$, since $V^{(B)}$ can be any smooth function that keeps $L$ non-degenerate and $h$ is some positive real number. However, assuming that $(q_{1,L_d},p_{1,L_d})=(q_{1,H_d^+},p_{1,H_d^+})$, will in general imply $V^{(B)}(q_A(q_0,q_{1,L_d},t)) \neq V^{(B)}(q_A(q_0,p_{1,H_d^+},t))$. To show this last claim, first note that all we need to show is that $q_A(q_0,q_{1,L_d},t) \neq q_A(q_0,p_{1,H_d^+},t)$. This inequality holds, since, as can be seen from \eqref{p0}, in general the map defined by $L_d^{(A),E}$ is not the same as the map defined by $L_d$. Therefore, the contradiction is complete, and in general, the maps defined by $L_d$ and $H_d^+$ are not the same map. However, both of these maps are self-adjoint.
\begin{theorem}
Assuming the flow associated with $L^{(A)}$ is time-reversible, then both methods, defined  respectively by \eqref{Leqtns} and \eqref{Heqtns}, are symmetric methods.
\end{theorem}
\begin{proof}
The discrete Lagrangian associated with \eqref{Leqtns} is given by,
\begin{align*}
L_d(q_0,q_1;h) = L_d^{(A),E}(q_0,q_1;h) - \epsilon \int_0^h V^{(B)}(q^A(q_0,q_1,t)) dt.
\end{align*}
The adjoint of the discrete Lagrangian is given by,
\begin{align*}
(L_d(q_0,q_1;h))^* &= -L_d(q_1,q_0;-h) \\
&= -L_d^{(A),E}(q_1,q_0;-h) + \epsilon \int_0^{-h} V^{(B)}(q^A(q_1,q_0,t)) dt \\
&= -L_d^{(A),E}(q_1,q_0;-h) - \epsilon \int_0^h V^{(B)}(q^A(q_1,q_0,t)) dt \\
&= L_d^{(A),E}(q_0,q_1;h) - \epsilon \int_0^h V^{(B)}(q^A(q_0,q_1,t)) dt \\
&= L_d(q_0,q_1;h).
\end{align*}
The third equality comes from the time-reversibility of the flow associated with $L^{(A)}$, and the fourth equality uses that property together with the fact that the exact discrete Lagrangian is self-adjoint. \\
The discrete right Hamiltonian associated with \eqref{Heqtns} is given by,
\begin{align*}
H_d^+(q_0,p_1;h) = H_d^{(A),+,E}(q_0,p_1;h) + \epsilon \int_0^h V^{(B)}(q_A(q_0,p_1,t)) dt.
\end{align*}
The adjoint of the discrete right Hamiltonian is given by,
\begin{align*}
(H_d^+)^*(p_0,q_1;h) &= -H_d^+(q_1,p_0;-h) \\
&= -H_d^{(A),+,E}(q_1,p_0;-h) - \epsilon \int_0^{-h} V^{(B)}(q_A(q_1,p_0,t)) dt \\
&= -H_d^{(A),+,E}(q_1,p_0;-h) + \epsilon \int_0^h V^{(B)}(q_A(q_1,p_0,t)) dt \\
&= H_d^{(A),-,E}(p_0,q_1;h) + \epsilon \int_0^h V^{(B)}(q_A(p_0,q_1,t)) dt \\
&= H_d^-(p_0,q_1;h),
\end{align*}
where the third equality comes from the time-reversibility of the flow associated with $H^{(A)}$, and the fourth equality uses that property together with the fact that the exact discrete Hamiltonian is self-adjoint.
\end{proof}

How do these respective maps differ? To gain insight into this question we now turn to numerical experimentation.
\subsection{Numerical Results}
Consider a Hamiltonian of the form,
\begin{equation}
H(q,p) = \frac{1}{2}(p^2+q^2) + \frac{\epsilon}{3} q^3,
\end{equation}
which is the Hamiltonian for a nonlinearly perturbed harmonic oscillator. The corresponding averaged Lagrangian is given by
\begin{equation}
L_d(q_0,q_1,h) = \int_0^h \frac{1}{2}(\dot{q}_A(t)^2 - q_A(t)^2) dt - \int_0^h \frac{\epsilon}{3} q_A(t)^3 dt,
\end{equation}
where $(q_A(t),\dot{q}_A(t))$ is the solution corresponding to the Lagrangian $L^{(A)}(q,\dot{q}) = \frac{1}{2}(\dot{q}^2 - q^2)$ with boundary conditions $(q_0,q_1)$. Analogously, the averaged Hamiltonian is given by
\begin{equation}
H_d^+(q_0,p_1,h) = p_1^Tq_A(h) - \int_0^h \frac{1}{2}(p_A(t)^2 - q_A(t)^2) dt + \frac{\epsilon}{3} \int_0^h q_A(t)^3 dt,
\end{equation}
where $(q_A(t),p_A(t))$ is the solution corresponding to the Hamiltonian $H^{(A)}(q,p)$ with boundary conditions $(q_0,p_1)$.
Applying the discrete right and left Legendre transforms implicitly defines the discrete Hamiltonian map for $L_d(q_0,q_1,h)$ and  the discrete right Hamiltonian map for $H_d^+(q_0,p_1,h)$, which yields the respective one-step methods. Numerical simulations were run over a time-span from 0 to 10000 or the nearest integer value to 10000 for the respective time-step. The initial conditions are given by $(q_0,p_0) = (1,0)$.
\\
\begin{figure}
  \includegraphics[width=0.9\textwidth]{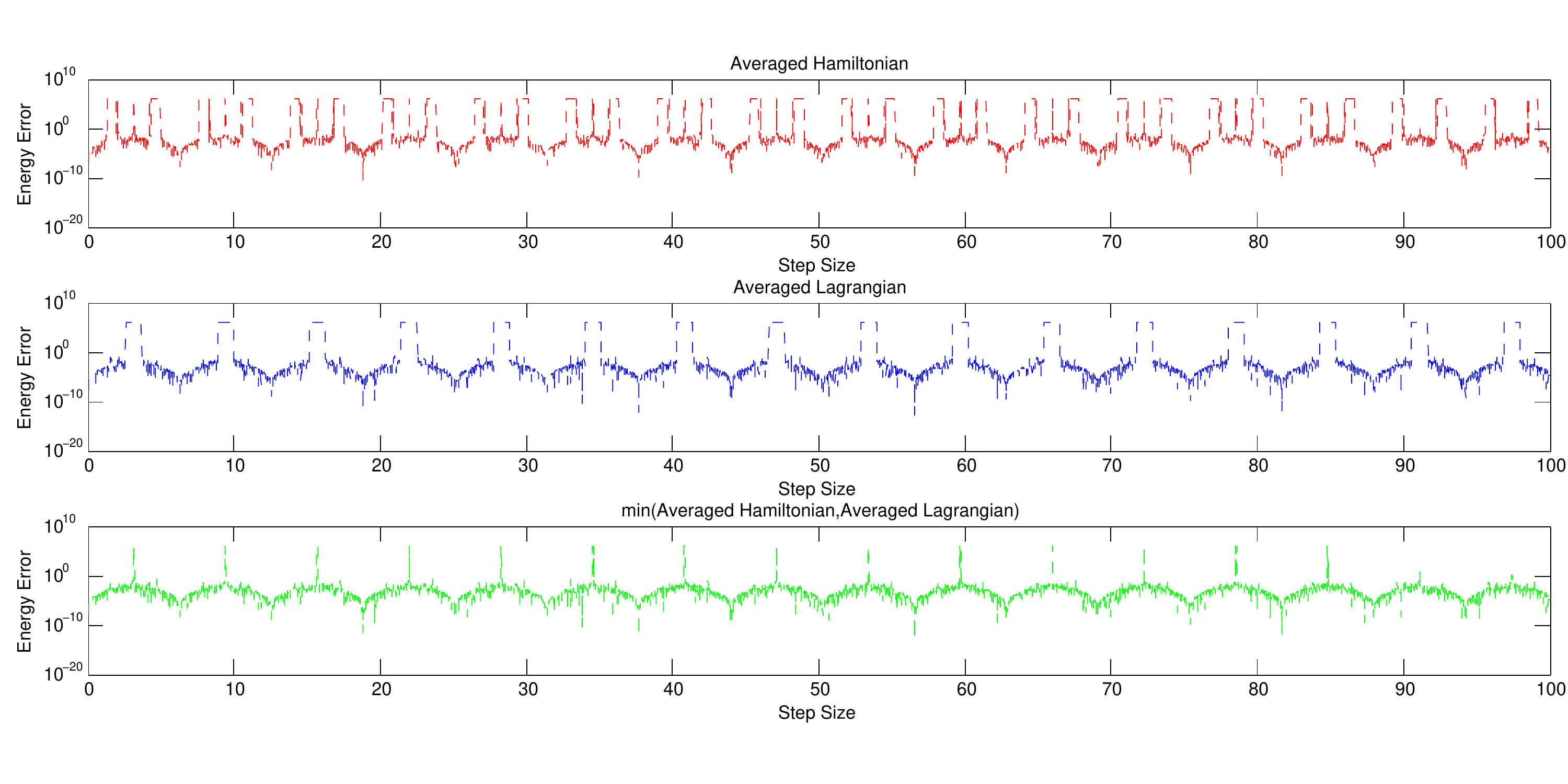}
  \caption
   {Three plots of step size versus energy error with fixed $\epsilon = 0.1$. The first plot corresponds to the averaged Hamiltonian, and it suffers from numerical resonance around odd integer multiples of $\frac{\pi}{2}$ and exactly at odd multiples $\pi$. The second plot corresponds to the averaged Lagrangian which suffers from numerical resonance around odd multiples of $\pi$. The last plot takes the minimum error of the respective methods.}
   \label{Avg3Plot}
\end{figure}
\begin{figure}
  \includegraphics[width=0.9\textwidth]{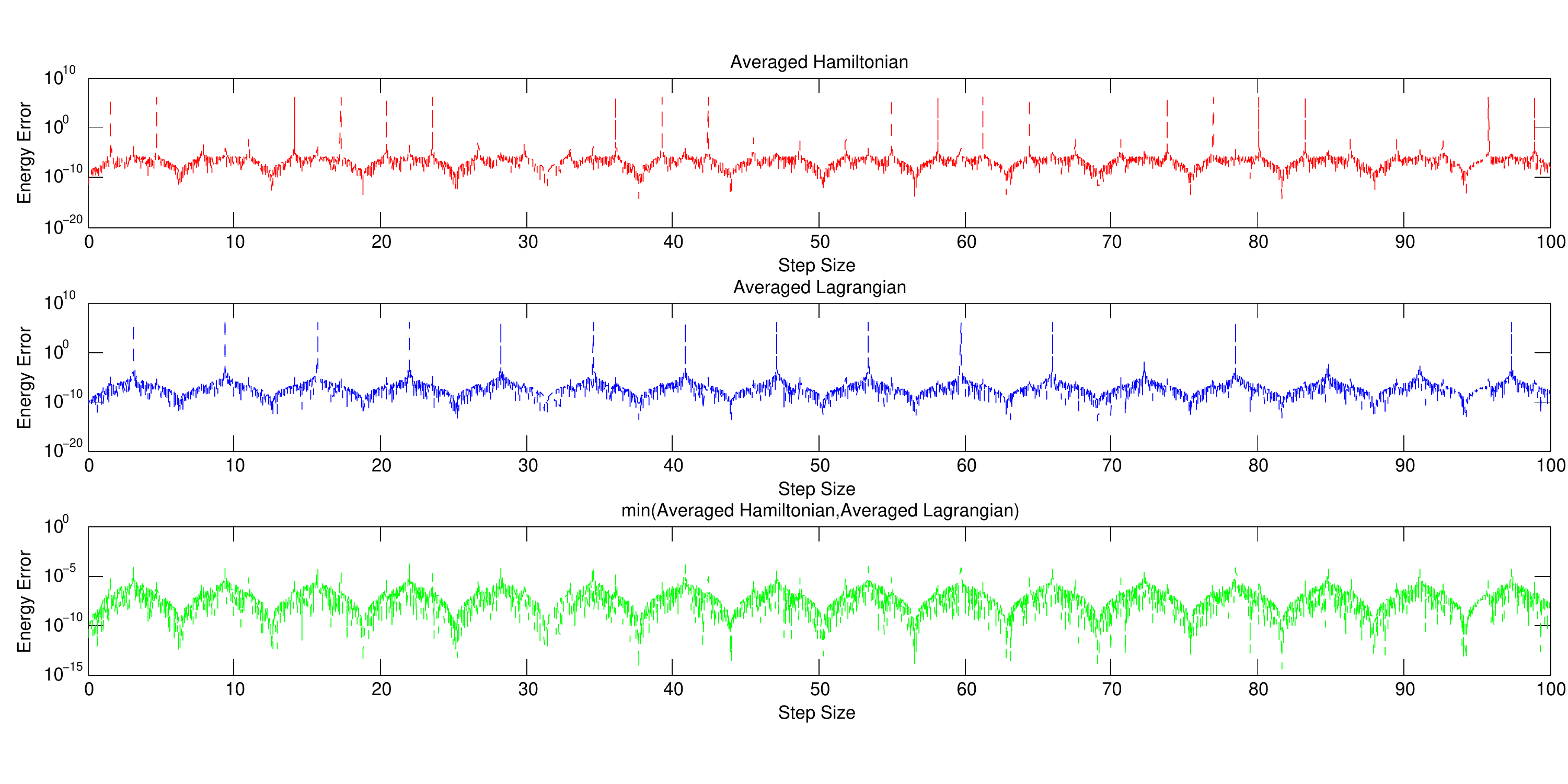}
  \caption 
   {Three plots of step size versus energy error with fixed $\epsilon = 0.001$. The first plot corresponds to the averaged Hamiltonian, and it suffers from numerical resonance at some odd integer multiples of $\frac{\pi}{2}$. The second plot corresponds to the averaged Lagrangian which suffers from numerical resonance around odd multiples of $\pi$. The last plot takes the minimum error of the respective methods.}
   \label{Avg3Plot2}
\end{figure}

Figures \ref{Avg3Plot} and \ref{Avg3Plot2} show plots of the energy error versus step size for two different values of $\epsilon$. Both figures demonstrate that the discrete Lagrangian and discrete right Hamiltonian have numerical resonance issues that are in some sense dual. The discrete Lagrangian exhibits excessive numerical resonance for step sizes near odd multiples of $\pi$, while the discrete right Hamiltonian exhibits excessive numerical resonance for step sizes near odd multiples of $\frac{\pi}{2}$. It should be noted that the arbitrary value of $10^6$ was substituted for output that was either near infinite or NaN. What is particularly striking is that the occurence of the numerical resonance is intimately connected to the corresponding boundary-values for each generating function. 
\\

To make the previous statement precise let us examine the unperturbed model. Consider the unperturbed harmonic oscillator boundary-value problem,
\begin{equation}
\ddot{q}(t) + q(t) = 0, \qquad q(0)=q_0, \ q(h)=q_1.
\end{equation}
Analytically, the boundary-value problem is not well-posed when $h$ is an integer multiple of $\pi$. Introducing round-off error into the picture only increases the region of instability around integer multiples of $\pi$. The energy error plot of the averaged Lagrangian (see Figures \ref{Avg3Plot} and \ref{Avg3Plot2}) for the perturbed harmonic oscillator exhibits excessive round-off error around similar values of $h$. Recall that the exact discrete Lagrangian is given by,
\begin{equation}
L_d^E(q_0,q_1;h)=\int_0^h L(q_{01}(t),\dot q_{01}(t)) dt,\label{exact_Ld_Jacobi}
\end{equation}
where $q_{01}(0)=q_0$, $q_{01}(h)=q_1$, and $q_{01}(t)$ satisfies the Euler--Lagrange equation in the time interval $(0,h)$. Thus, it is ultimately defined in terms of such a boundary-value problem. The integrator obtained from the exact discrete Lagrangian is given by,
\begin{align*}
q_1 &= q_0\cos(h) + p_0\sin(h), \\
p_1 &= q_1\cot(h) - q_0\csc(h).
\end{align*}
Noting that $\cot(h)$ and $\csc(h)$ both involve dividing by $\sin(h)$, we expect increased round-off error around values of $h$ that are integer multiples of $\pi$.

Similarly, the exact discrete right Hamiltonian is given by,
\begin{equation}
H_d^{+,E}(q_0,p_1;h)=p_1^Tq_1 - \int_0^h [p_{01}(t)^T\dot{q}_{01}(t) - H(q_{01}(t),p_{01}(t))] dt,
\end{equation}
where $q_{01}(0)=q_0$, $p_{01}(h)=p_1$, and $(q_{01}(t),p_{01}(t))$ satisfies Hamilton's equations in the time interval $(0,h)$. This is related to the unperturbed harmonic oscillator boundary-value problem given by,
\begin{equation}
\dot{q}(t)=p(t), \qquad \dot{p}(t)=-q(t), \qquad q(0)=q_0, \ p(h)=p_1.
\end{equation}
This boundary-value problem is not well-posed for values of $h$ that are odd multiples of $\frac{\pi}{2}$. The energy error plot of the averaged Hamiltonian for the perturbed harmonic oscillator also exhibits round-off error around these values of $h$. The integrator obtained from the exact discrete right Hamiltonian for the unperturbed harmonic oscillator is given by,
\begin{align*}
p_1 &= p_0\cos(h) - q_0\sin(h), \\
q_1 &= p_1\tan(h) + q_0\sec(h).
\end{align*}
Noting that the method involves $\tan(h)$ and $\sec(h)$, we expect increased round-off error around odd multiples of $\frac{\pi}{2}$.
\\

Both of the integrators given by the exact discrete Lagrangian and the exact discrete right Hamiltonian have been implemented for the harmonic oscillator with initial conditions $(q_0,p_1) = (1,0)$ over the time interval $[0,10000]$, and the energy error is shown in Figure \ref{AvgExact}. Note the jump in round-off error corresponding to values of $h$ that are odd multiples of $\pi$ (for the discrete Lagrangian) and odd multiples of $\frac{\pi}{2}$ (for the discrete right Hamiltonian).
\\

\begin{figure}
  \includegraphics[width=0.9\textwidth]{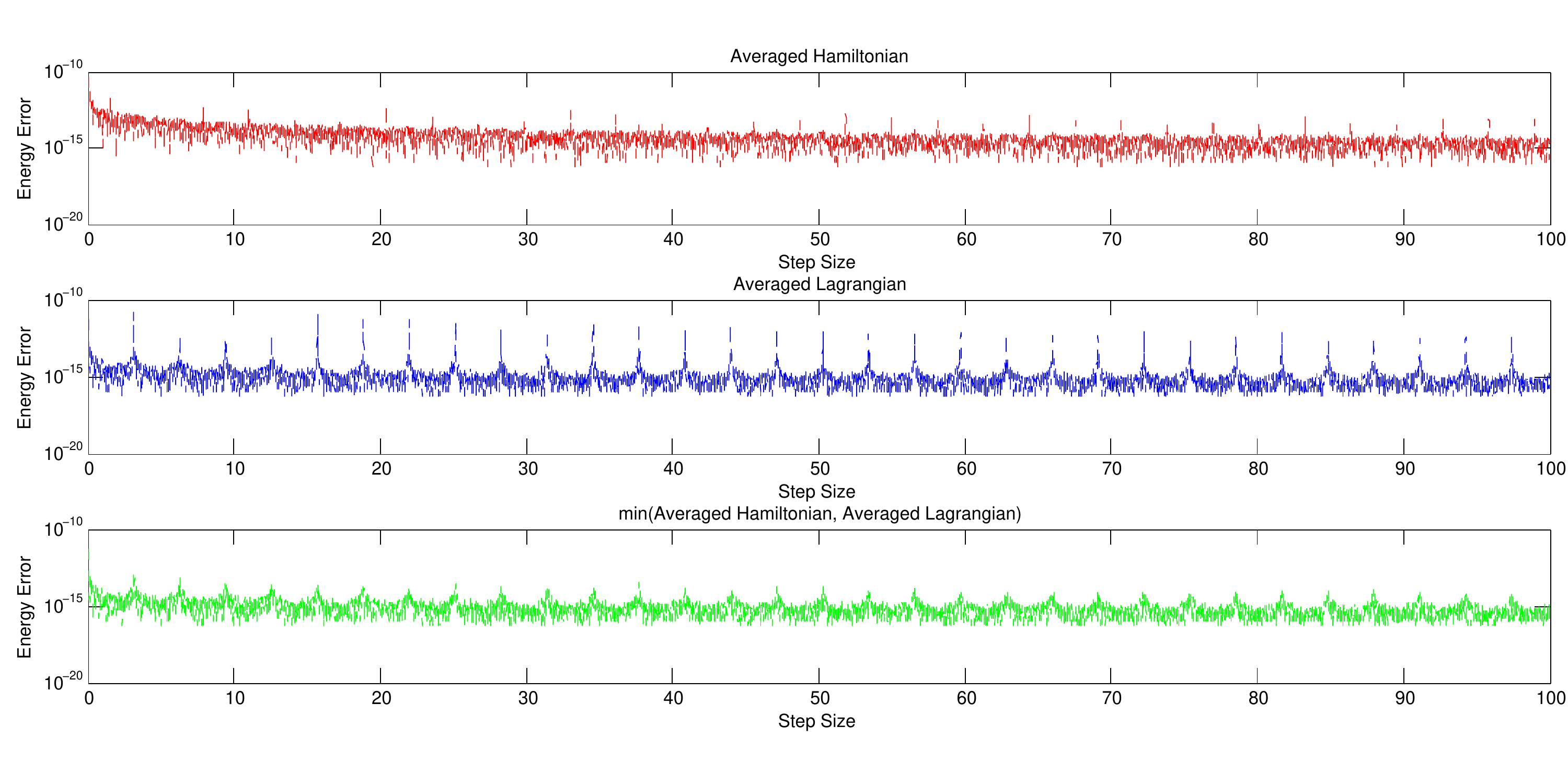}
  \caption
   {The first plot is the energy error versus step size for the exact discrete right Hamiltonian applied to the harmonic oscillator. The second plot shows the energy error versus step size for the exact discrete Lagrangian, while the third plot takes the minimum of the energy error from either method.}
   \label{AvgExact}
\end{figure}

Thus, in this particular case, we can conclude that the difference between the symplectic maps generated by the respective discrete Lagrangian and discrete Hamiltonian is a matter of numerical conditioning, which is inherited from the underlying ill-posedness of the associated boundary-value problem.
\\

Now this by no means provides a rigorous analysis of the numerical resonances, nor does it fully explain all of the resonance effects, but it does provide motivation and insight into the numerical differences between the discrete Lagrangian and discrete right Hamiltonian. A more in-depth analysis might be provided by applying something similar to modulated Fourier expansions (see \cite{HaLu2001,HaLu2012}, and Chapter XIII of \cite{HaLuWa2006}). Modulated Fourier expansions are particularly well-suited for oscillatory problems when large step sizes are sought. The standard backward error analysis relies on $h\omega \to 0$, which is not the case for high oscillatory problems when seeking large step sizes. Modulated Fourier expansions can provide a tool for deriving many of the same results as backward error analysis, such as long-term energy preservation. Furthermore, it can be quite useful for examining the step sizes that lead to excessive numerical resonance. However, it should be noted that while modulated Fourier expansions have been used quite successfully to analyze explicit trigonometric integrators, it is not quite as clear how easily it can deal with implicit integrators such as those obtained from the discrete averaged Lagrangian and discrete averaged Hamiltonian.

\subsection{Fermi--Pasta--Ulam Simulation}
The previous section showed important differences between discrete Lagrangians and discrete Hamiltonians when applied to harmonic oscillator problems. This difference could be interpreted as being related to the conditioning of the respective boundary-value problem for the Lagrangian and Hamiltonian. Is this difference visible in highly oscillatory phenomenon? For this we turn to the Fermi--Pasta--Ulam (FPU) problem (see \cite{FPU55,HaLuWa2006}). This is a model of mass points connected together by an alternating series of stiff harmonic and soft nonlinear springs, where the first and last mass points are held fixed. Denoting the displacement of the mass points by $q_1,\dots, q_{2m}$ and the velocites $\dot{q}_i=p_i$, then the associated Hamiltonian is,
\begin{align*}
H(p,q)=\frac{1}{2}\sum_{i=1}^m(p_{2i-1}^2+p_{2i}^2)+\frac{\omega^2}{4}\sum_{i=1}^m(q_{2i}-q_{2i-1})^2+\sum_{i=0}^m(q_{2i+1}-q_{2i})^4.
\end{align*}
Under an appropriate change of variables, the total oscillatory energy of the stiff springs is nearly constant. For our simulation $\omega=50$ and $m=3$, so there will be 3 stiff springs whose sum of oscillatory energy should remain close to constant. Figure \ref{FPU} is a plot of the oscillatory energies of the stiff springs approximated by the various numerical integrators.
\\

The following simulations used a Lagrangian Taylor variational integrator and a  Hamiltonian Taylor variational integrator. Both were constructed using the trapezoid quadrature rule and a zeroth-order Taylor method. The Lagrangian construction resulted in the method,
\begin{align*}
q_1 &= q_0 +hM^{-1}p_0-\frac{h^2}{2}M^{-1}\nabla V(q_0), \\
p_1 &= p_0 - \frac{h}{2}[\nabla V(q_0) + \nabla V(q_1)],
\end{align*}
which is better known as the St\"ormer--Verlet method. The Hamiltonian construction resulted in the method,
\begin{align*}
q_1 &= q_0 +hM^{-1}p_0-\frac{h^2}{2}M^{-1}\nabla V(q_0), \\
p_1 &= p_0 - \frac{h}{2}[\nabla V(q_0) + \nabla V(q_0+hM^{-1}p_1)],
\end{align*}
which is not St\"ormer--Verlet. This method is in fact implicit, while St\"ormer--Verlet is explicit and symmetric. In addition, the implicit-explicit method (IMEX) was used, as it has been shown in \cite{McStern2014} to be optimal, in a certain sense, among all modified trigonometric integrators for highly oscillatory problem such as the FPU model. This numerical method essentially mixes the midpoint method for the fast, linear part and the St\"ormer--Verlet method for the slow, nonlinear part.

\begin{figure}
  \includegraphics[width=0.9\textwidth]{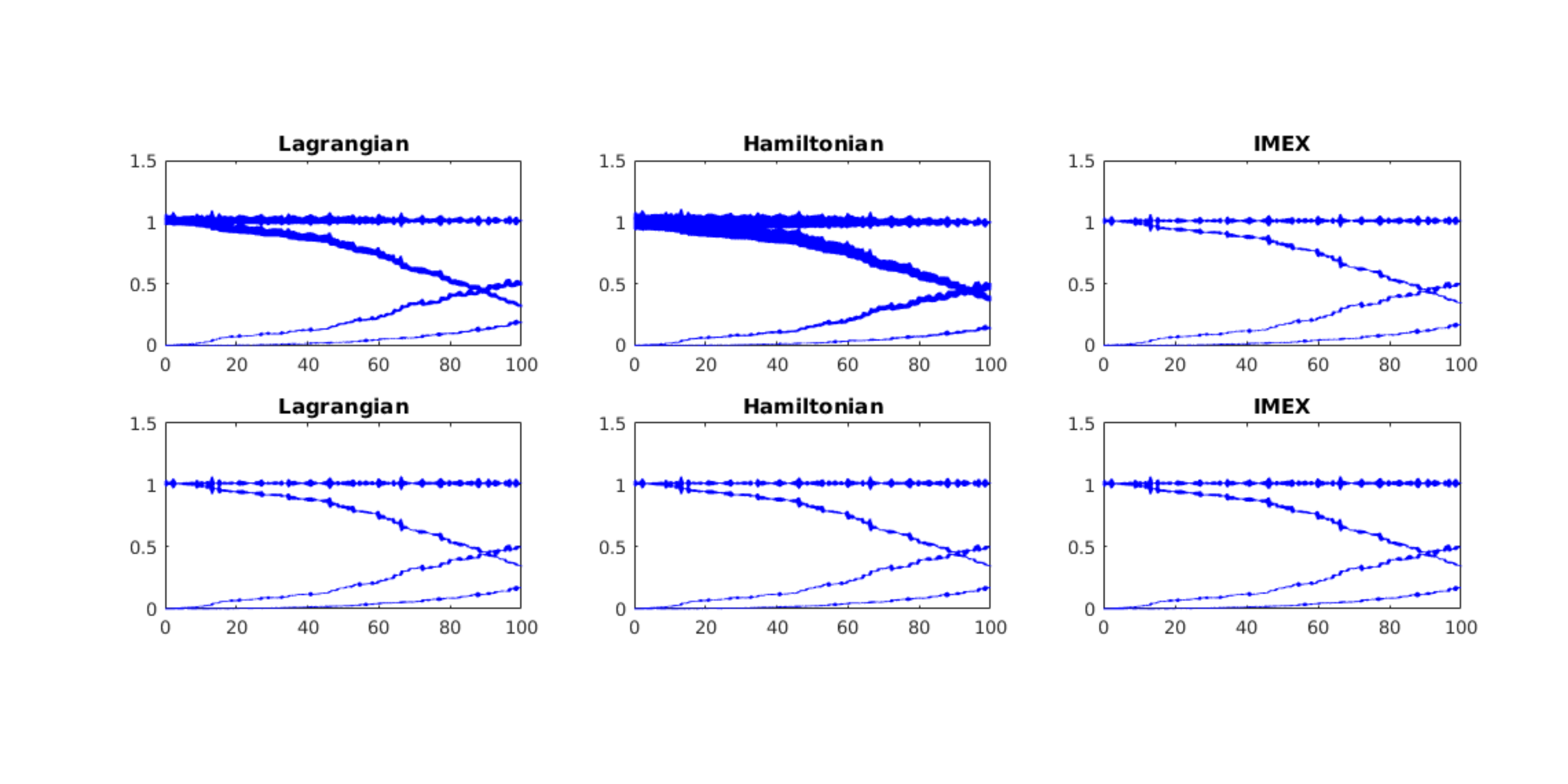}
  \caption
   {The first row of plots involved a step size of $h=0.01$, while the bottom row used $h=0.001$. The Lagrangian derived integrator can be seen to outperform the Hamiltonian derived integrator. The IMEX method performs the best among the three methods.}
   \label{FPU}
\end{figure}

The plots clearly show that the Lagrangian method (St\"ormer--Verlet) outperformed the Hamiltonian method. Even though the Hamiltonian method was implicit, the symmetry of the St\"ormer--Verlet method may be the more important property for highly-oscillatory problems (see \cite{HaLuWa2006}). Variational integrators derived from an approximation scheme that involve a one-step method, applied to the boundary-value problem formulation of the Lagrangian or Hamiltonian, are only likely to be symmetric when derived from a Lagrangian formulation. This can be seen to come from the independent variables associated with a Type I, II, or III generating function. The boundary values $(q_0,q_1)$ lend themselves to symmetry more readily than $(q_0,p_1)$ or $(q_1,p_0)$. That being said, other approximation schemes, such as the averaging methods of the previous section, can generate symmetric integrators using either formulation.

\section{Conclusion}
Error analysis and symmetry results have now been extended to cover discrete Hamiltonian variational integrators. Furthermore, many examples have been presented indicating that the properties of variational integrators are dependent on both the approximation scheme used in constructing the generating function and the type of generating function being approximated.

This paper indicates that the class of variational integrators generated using the Hamiltonian formulation are not necessarily equivalent to the ones obtained from the Lagrangian formulation, and it would therefore be of interest to continue developing methods based on the discrete Hamiltonian variational integrator formulation. In particular, the results presented suggest that further work remains to be done to better understand the circumstances under which it is preferable to favor one approach over the other.

\section*{Acknowledgements}
This research has been supported in part by NSF under grants DMS-1010687, CMMI-1029445, DMS-1065972, CMMI-1334759, DMS-1411792, DMS-1345013.

\bibliography{vi_review}
\bibliographystyle{plainnat}
\end{document}